\documentclass[reqno, 12pt]{amsart}
\usepackage{amssymb}
\usepackage{amsfonts}
\usepackage{srcltx}
\usepackage{enumerate}
\usepackage{cite}

\newtheorem{theorem}{Theorem}[]
\newtheorem{lemma}[]{Lemma}

\theoremstyle{definition}

\newtheorem{example}{Example}

\theoremstyle{remark}
\newtheorem*{remark}{Remark}

\newtheorem*{problem}{Problem}

\DeclareMathOperator{\dist}{dist}
\DeclareMathOperator{\diam}{diam}

\DeclareMathOperator{\Fix}{Fix}

\begin{document}
\title[Uniformly Lipschitzian semigroups]{H\"{o}lder continuous retractions
and amenable semigroups of uniformly Lipschitzian mappings in Hilbert spaces}
\author[A. Wi\'{s}nicki]{Andrzej Wi\'{s}nicki}

\begin{abstract}
Suppose that $S$ is a left amenable semitopological semigroup. We prove that
if $\mathcal{S}=\left\{ T_{t}:t\in S\right\} $ is a uniformly $k$%
-Lipschitzian semigroup on a bounded closed and convex subset $C$ of a
Hilbert space and $k<\sqrt{2},$ then the set of fixed points of $\mathcal{S}$
is a H\"{o}lder continuous retract of $C$. This gives a qualitative
complement to the Ishihara-Takahashi fixed point existence theorem.
\end{abstract}

\dedicatory{
 Dedicated to Professors Kazimierz Goebel and Lech G\'{o}rniewicz
 \\ on the occasion of their 70th birthdays}

\address{Andrzej Wi\'{s}nicki, Institute of Mathematics, Maria Curie-Sk\l %
odowska University, 20-031 Lublin, Poland}
\email{awisnic@hektor.umcs.lublin.pl}
\maketitle

\section{Introduction}

Let $C$ be a nonempty bounded closed and convex subset of a Banach space $X$%
. A mapping $T:C\rightarrow C$ is said to be nonexpansive if
\begin{equation*}
\left\Vert Tx-Ty\right\Vert \leq \left\Vert x-y\right\Vert
\end{equation*}%
for $x,y\in C$. The study of the structure of the fixed-point set $\Fix %
T=\left\{ x\in C:Tx=x\right\} $ was initiated by R. Bruck (cf. \cite{Br1,
Br2}) who proved that if $T$ has a fixed point in every nonempty closed
convex subset of $C$ which is invariant under $T$, and $C$ is convex and
weakly compact or separable, then $\Fix T$ is a nonexpansive retract of $C$
(i.e., there exists a nonexpansive mapping $R:C\rightarrow \Fix T$ such that
$R_{\mid \Fix T}=I$). Notice that the set of fixed points of a $k$-Lipschitz
mapping can be very irregular for any $k>1$. Indeed, let $F$ be a nonempty
closed subset of $C$, $z\in F$, $\varepsilon >0$, and define
\begin{equation*}
Tx=x+\varepsilon \dist\left( x,F\right) \left( z-x\right) ,\ x\in C\text{.}
\end{equation*}%
Then the Lipschitz constant of $T$ tends to $1$ if $\varepsilon \rightarrow
0 $, but $\Fix T=F$ (see, e.g., \cite{SeWi}).

In 1973, Goebel and Kirk \cite{GoKi} introduced the class of uniformly
Lipschitzian mappings and proved a fixed point theorem which was later
studied by several authors. A mapping $T:C\rightarrow C$ is said to be $k$%
-uniformly Lipschitzian if%
\begin{equation*}
\left\Vert T^{n}x-T^{n}y\right\Vert \leq k\left\Vert x-y\right\Vert
\end{equation*}%
for every $x,y\in C$ and $n\in \mathbb{N}$. Recall that a Banach space $X$
is uniformly convex if $\delta _{X}(\varepsilon )>0$ for $\varepsilon >0,$
where $\delta _{X}:[0,2]\rightarrow \lbrack 0,1]$ is the modulus of
convexity of $X$ defined by

\begin{equation*}
\delta _{X}(\varepsilon )=\inf \{1-\left\Vert \frac{x+y}{2}\right\Vert
:\left\Vert x\right\Vert \leq 1,\left\Vert y\right\Vert \leq 1,\left\Vert
x-y\right\Vert \geq \varepsilon \}.
\end{equation*}

\begin{theorem}[cf. \protect\cite{GoKi}]
{\label{GoKi}} Let $X$ be a uniformly convex Banach space with the modulus
of convexity $\delta _{X}$ and let $C$ be a bounded, closed and convex
subset of $X$. Suppose $T:C\rightarrow C$ is uniformly $k$-Lipschitzian and
\begin{equation}
k\left( 1-\delta _{X}\left( \frac{1}{k}\right) \right) <1.  \label{gk}
\end{equation}%
Then $T$ has a fixed point in $C$.
\end{theorem}

It was shown in \cite{SeWi} that under the assumptions of Theorem \ref{GoKi}%
, the fixed-point set $\Fix T$ is a (continuous) retract of $C.$ Recently, G%
\'{o}rnicki (cf. \cite{GoJ, Gosem}) proved several structural results
concerning uniformly Lipschitzian mappings but many questions remain open.
In \cite{PeFe}, P\'{e}rez Garc\'{\i}a and Fetter Nathansky gave conditions
under which $\Fix T$ is a H\"{o}lder continuous retract and applied them to
the study of n-periodic mappings in Hilbert spaces.

Notice that in a Hilbert space a solution to (\ref{gk}) gives $k<\sqrt{5}/2.$
Lifschitz \cite{Li} improved this estimation and showed that in a Hilbert
space a uniformly $k$-Lipschitzian mapping with $k<\sqrt{2}$ has a fixed
point. The Lifschitz theorem was generalized to uniformly $k$-Lipschitzian
semigroups in \cite{DoRa, IsTa, MiTa}. The aim of this note is to give a
qualitative complement to the above results in the case of left amenable
(semitopological) semigroups which partially extends a result of G\'{o}%
rnicki (cf. \cite[Cor. 14]{Gosem}). We show that if $\mathcal{S}=\left\{
T_{t}:t\in S\right\} $ is a uniformly $k$-Lipschitzian semigroup on $C$ and $%
k<\sqrt{2},$ then $\Fix\mathcal{S}=\bigcap_{t\in S}\left\{ x\in
C:T_{t}x=x\right\} ,$ the set of (common) fixed points of $\mathcal{S},$ is
a H\"{o}lder continuous retract of $C.$

\section{Fixed point theorem}

We start with the following variant of a well known result (see, e.g., \cite[%
Prop. 1.10]{BeLi}, \cite[Lemma 2.1]{PeFe}).

\begin{lemma}
\label{holder}Let $(X,d)$ be a complete bounded metric space and let $%
L:X\rightarrow X$ be a $k$-Lipschitz mapping. Suppose there exist $0<\gamma
<1$ and $c>0$ such that $d(L^{n+1}x,L^{n}x)\leq c\gamma ^{n}$ for every $%
x\in X$ and $n\in \mathbb{N}.$ Then $Rx=\lim_{n\rightarrow \infty }L^{n}x$
is a H\"{o}lder continuous mapping on $X.$
\end{lemma}

\begin{proof}
We may assume that $\diam X=1$. Notice that for every $x\in X$ and $n,m\in
\mathbb{N},$%
\begin{equation*}
d(L^{n+m}x,L^{n}x)\leq c\frac{\gamma ^{n}}{1-\gamma }.
\end{equation*}%
Hence, for every $x,y\in X,$%
\begin{equation*}
d(Rx,Ry)\leq d(Rx,L^{n}x)+d(L^{n}x,L^{n}y)+d(L^{n}y,Ry)\leq 2c\frac{\gamma
^{n}}{1-\gamma }+k^{n}d(x,y).
\end{equation*}%
Take $\alpha <1$ such that $k\leq \gamma ^{1-\alpha ^{-1}}$ and fix $x,y\in
X,$ $x\neq y.$ Then, there exist $n\in \mathbb{N}$ and $0<r\leq 1$ such that
$\gamma ^{n-r}=d(x,y)^{\alpha }.$ Furthermore, $k^{n-1}\leq (\gamma
^{1-\alpha ^{-1}})^{n-r}.$ It follows that
\begin{equation*}
d(Rx,Ry)\leq 2c\frac{\gamma ^{n-r}}{1-\gamma }+k(\gamma ^{n-r})^{1-\alpha
^{-1}}d(x,y)=(\frac{2c}{1-\gamma }+k)d(x,y)^{\alpha }.
\end{equation*}
\end{proof}

Let $S$ be a semigroup and $\ell ^{\infty }(S)$ the Banach space of bounded
real valued functions on $S$ with the supremum norm. For $s\in S$ and $f\in
\ell ^{\infty }(S),$ we define elements $l_{s}f$ and $r_{s}f$ in $\ell
^{\infty }(S)$ by%
\begin{equation*}
l_{s}f(t)=f(st),\ r_{s}f(t)=f(ts)
\end{equation*}%
for every $t\in S.$ An element $\mu $ of $(\ell ^{\infty }(S))^{\ast }$ is
said to be a mean on $X$ if $\left\Vert \mu \right\Vert =\mu (I_{S})=1,$
where $I_{S}(t)=1$ for all $t\in S.$ It is well known that $\mu $ is a mean
if and only if
\begin{equation*}
\inf_{t\in S}f(t)\leq \mu (f)\leq \sup_{t\in S}f(t)
\end{equation*}%
for each $f\in \ell ^{\infty }(S).$ A mean $\mu $ on $\ell ^{\infty }(S)$ is
said to be left (resp. right) invariant if $\mu (l_{s}f)=\mu (f)$ (resp. $%
\mu (r_{s}f)=\mu (f)$) for each $s\in S$ and $f\in \ell ^{\infty }(S).$ A
Banach limit is a special case of an invariant mean, where $S$ equals the
semigroup of natural numbers. A semigroup $S$ is called left (resp. right)
amenable if there exists a left (resp. right) invariant mean on $\ell
^{\infty }(S).$ The term \textquotedblleft amenable
semigroup\textquotedblright\ was coined by M. M. Day in his celebrated paper
\cite{Da}.

Let $C$ be a bounded closed convex subset of a Hilbert space $H.$ For
simplicity, we consider only a real Hilbert space. A family $\mathcal{S}%
=\left\{ T_{t}:t\in S\right\} $ of mappings from $C$ into $C$ is said to be
a uniformly $k$-Lipschitzian semigroup on $C$ if $T_{ts}x=T_{t}\,T_{s}x$ and
\begin{equation*}
\left\Vert T_{t}x-T_{t}y\right\Vert \leq k\left\Vert x-y\right\Vert
\end{equation*}%
for all $t,s\in S$ and $x\in C.$

The following construction is well known (see, e.g., \cite{Ar, Ta}). Let $%
\mu $ be a mean on $\ell ^{\infty }(S),x\in C$, and consider a functional $%
F(y)=\int \left\langle T_{t}x,y\right\rangle d\mu (t),y\in H,$ where $\int
\left\langle T_{t}x,y\right\rangle d\mu (t)$ denotes the value of $\mu $ at
the function $t\rightarrow \left\langle T_{t}x,y\right\rangle .$ It is not
difficult to see that $F$ is linear and continuous since%
\begin{equation*}
\left\vert \int \left\langle T_{t}x,y\right\rangle d\mu (t)\right\vert \leq
\sup_{t\in S}\left\vert \left\langle T_{t}x,y\right\rangle \right\vert \leq
\sup_{t\in S}\left\Vert T_{t}x\right\Vert \left\Vert y\right\Vert .
\end{equation*}%
By the Riesz theorem, there exists a unique element $\bar{x}$ such that
\begin{equation}
\int \left\langle T_{t}x,y\right\rangle d\mu (t)=\left\langle \bar{x}%
,y\right\rangle  \label{riesz}
\end{equation}%
for every $y\in H.$ Furthermore, by the separation theorem, $\bar{x}\in C.$
Thus we obtain a mapping $\bar{T}_{\mu }:C\rightarrow C$ by putting $\bar{T}%
_{\mu }x=\bar{x}.$ Notice that if a semigroup $\mathcal{S}=\left\{
T_{t}:t\in S\right\} $ is uniformly $k$-Lipschitzian, then%
\begin{equation*}
\left\langle \bar{T}_{\mu }x-\bar{T}_{\mu }y,v\right\rangle =\int
\left\langle T_{t}x-T_{t}y,v\right\rangle d\mu (t)\leq \sup_{t\in
S}\left\Vert T_{t}x-T_{t}y\right\Vert \left\Vert v\right\Vert \leq
k\left\Vert x-y\right\Vert \left\Vert v\right\Vert
\end{equation*}%
for every $x,y\in C,v\in H$, and hence
\begin{equation*}
\left\Vert \bar{T}_{\mu }x-\bar{T}_{\mu }y\right\Vert \leq k\left\Vert
x-y\right\Vert
\end{equation*}%
for every $x,y\in C$, i.e., $\bar{T}_{\mu }$ is $k$-Lipschitz.

The above notions may be generalized to the case of semitopological
semigroups. Recall that $S$ is a topological semigroup if there exists a
Hausdorff topology on $S$ such that the mapping $S\times S\ni
(s,t)\rightarrow st$ is (jointly) continuous. A semigroup $S$ is
semitopological if for each $t\in S$ the mappings $S\ni s\rightarrow ts$ and
$S\ni s\rightarrow st$ are continuous. Notice that every semigroup can be
equipped with the discrete topology and then it is called a discrete
semigroup. Let $X$ be a linear closed subspace of $\ell ^{\infty }(S)$
containing $I_{S}$ such that $l_{s}(X)\subset X$ and $r_{s}(X)\subset X$ for
each $s\in S.$ Let $X^{\ast }$ be its topological dual. Then $X$ is said to
be left (resp. right) amenable if there exists a left (resp. right)
invariant mean on $X.$

The following theorem is a qualitative version of (a slight generalization
of) Theorem 1 in \cite{IsTa}.

\begin{theorem}
\label{H}Let $C$ be a nonempty bounded closed convex subset of a Hilbert
space $H$ and $\mathcal{S}=\left\{ T_{t}:t\in S\right\} $ a uniformly $k$%
-Lipschitzian semigroup on $C.$ Suppose that $X$ is a left amenable subspace
of $\ell ^{\infty }(S)$ such that the functions $g(t)=\left\langle
T_{t}x,y\right\rangle $ and $h(t)=\left\Vert T_{t}x-y\right\Vert ^{2}$
belong to $X$ for each $x\in C$ and $y\in H.$ If $k<\sqrt{2},$ then $%
\mathcal{S}$ has a fixed point in $C$ and $\Fix\mathcal{S}=\bigcap_{t\in
S}\left\{ x\in C:T_{t}x=x\right\} $ is a H\"{o}lder continuous retract of $%
C. $
\end{theorem}

\begin{proof}
Let $\mu $ be a left invariant mean $\mu $ on $X.$ Then there exists a
mapping $\bar{T}_{\mu }:C\rightarrow C$ such that%
\begin{equation*}
\int \left\langle T_{t}x,y\right\rangle d\mu (t)=\left\langle \bar{T}_{\mu
}x,y\right\rangle
\end{equation*}%
for every $y\in H$ (see (\ref{riesz})). Following \cite[Th. 1]{IsTa}, fix $%
x_{0}\in C$ and put $x_{n+1}=\bar{T}_{\mu }x_{n},n=0,1,....$ Then%
\begin{equation*}
\left\Vert T_{t}x_{n}-y\right\Vert ^{2}=\left\Vert
T_{t}x_{n}-x_{n+1}\right\Vert ^{2}+\left\Vert x_{n+1}-y\right\Vert
^{2}+2\left\langle T_{t}x_{n}-x_{n+1},x_{n+1}-y\right\rangle ,
\end{equation*}%
\begin{equation*}
\int \left\langle T_{t}x_{n}-x_{n+1},x_{n+1}-y\right\rangle d\mu
(t)=\left\langle \bar{T}_{\mu }x_{n}-x_{n+1},x_{n+1}-y\right\rangle =0,
\end{equation*}%
and hence
\begin{equation*}
\int \left\Vert T_{t}x_{n}-y\right\Vert ^{2}d\mu (t)=\int \left\Vert
T_{t}x_{n}-x_{n+1}\right\Vert ^{2}d\mu (t)+\left\Vert x_{n+1}-y\right\Vert
^{2}
\end{equation*}%
for every $y\in H$ and $n=0,1,....$ It follows (writing $y=x_{n}$) that%
\begin{equation}
\int \left\Vert T_{t}x_{n}-x_{n+1}\right\Vert ^{2}d\mu (t)\leq \int
\left\Vert T_{t}x_{n}-x_{n}\right\Vert ^{2}d\mu (t).  \label{iq1}
\end{equation}%
Furthermore, taking $y=T_{s}x_{n+1},$ we have
\begin{align*}
\left\Vert T_{s}x_{n+1}-x_{n+1}\right\Vert ^{2}& =\int \left\Vert
T_{t}x_{n}-T_{s}x_{n+1}\right\Vert ^{2}d\mu (t)-\int \left\Vert
T_{t}x_{n}-x_{n+1}\right\Vert ^{2}d\mu (t) \\
& =\int \left\Vert T_{st}x_{n}-T_{s}x_{n+1}\right\Vert ^{2}d\mu (t)-\int
\left\Vert T_{t}x_{n}-x_{n+1}\right\Vert ^{2}d\mu (t) \\
& \leq (k^{2}-1)\int \left\Vert T_{t}x_{n}-x_{n+1}\right\Vert ^{2}d\mu (t)
\end{align*}%
for any $s\in S,$ since $\mu $ is left invariant and $\left\Vert
T_{st}x_{n}-T_{s}x_{n+1}\right\Vert \leq k\left\Vert
T_{t}x_{n}-x_{n+1}\right\Vert .$ Thus%
\begin{equation}
\int \left\Vert T_{s}x_{n+1}-x_{n+1}\right\Vert ^{2}d\mu (s)\leq
(k^{2}-1)\int \left\Vert T_{t}x_{n}-x_{n+1}\right\Vert ^{2}d\mu (t).
\label{iq2}
\end{equation}%
Combining (\ref{iq1}) with (\ref{iq2}) yields
\begin{equation*}
\int \left\Vert T_{t}x_{n+1}-x_{n+1}\right\Vert ^{2}d\mu (t)\leq
(k^{2}-1)\int \left\Vert T_{t}x_{n}-x_{n}\right\Vert ^{2}d\mu (t).
\end{equation*}%
for $n=0,1,....$ Hence%
\begin{align*}
\left\Vert \bar{T}_{\mu }^{n+1}x_{0}-\bar{T}_{\mu }^{n}x_{0}\right\Vert
^{2}& =\left\Vert x_{n+1}-x_{n}\right\Vert ^{2}\leq 2\int \left\Vert
x_{n+1}-T_{t}x_{n}\right\Vert ^{2}d\mu (t) \\
+2\int \left\Vert T_{t}x_{n}-x_{n}\right\Vert ^{2}d\mu (t)& \leq 4\int
\left\Vert T_{t}x_{n}-x_{n}\right\Vert ^{2}d\mu (t)\leq 4(k^{2}-1)^{n}\diam C
\end{align*}%
for every $x_{0}\in C$ and $n=0,1,....$ Since $\bar{T}_{\mu }$ is $k$%
-Lipschitz and $k<\sqrt{2}$, it follows from Lemma \ref{holder} that $%
Rx=\lim_{n\rightarrow \infty }\bar{T}_{\mu }^{n}x$ is H\"{o}lder continuous
on $C$. We show that $R$ is a retraction onto $\Fix\mathcal{S}.$ It is clear
that if $x\in \Fix\mathcal{S},$ then $Rx=\bar{T}_{\mu }x=x.$ Furthermore, it
follows from the generalized parallelogram law that for every $x,y,z\in H,$
\begin{align*}
\left\Vert x+y+z\right\Vert ^{2}& =\left\Vert x+y\right\Vert ^{2}+\left\Vert
y+z\right\Vert ^{2}+\left\Vert z+x\right\Vert ^{2}-\left\Vert x\right\Vert
^{2}-\left\Vert y\right\Vert ^{2}-\left\Vert z\right\Vert ^{2} \\
& \leq 2(\left\Vert x\right\Vert ^{2}+\left\Vert y\right\Vert
^{2})+2(\left\Vert y\right\Vert ^{2}+\left\Vert z\right\Vert
^{2})+2(\left\Vert z\right\Vert ^{2}+\left\Vert x\right\Vert ^{2}) \\
& -\left\Vert x\right\Vert ^{2}-\left\Vert y\right\Vert ^{2}-\left\Vert
z\right\Vert ^{2}=3(\left\Vert x\right\Vert ^{2}+\left\Vert y\right\Vert
^{2}+\left\Vert z\right\Vert ^{2}).
\end{align*}%
Therefore, for every $x\in C$ and $n=0,1,...,$%
\begin{eqnarray*}
\int \Vert T_{t}Rx-Rx\Vert ^{2}d\mu (t) &\leq &3\int (\left\Vert
T_{t}Rx-T_{t}\bar{T}_{\mu }^{n}x\right\Vert ^{2}+\left\Vert T_{t}\bar{T}%
_{\mu }^{n}x-\bar{T}_{\mu }^{n}x\right\Vert ^{2} \\
+\left\Vert \bar{T}_{\mu }^{n}x-Rx\right\Vert ^{2})d\mu (t) &\leq
&3(k^{2}+1)\left\Vert Rx-\bar{T}_{\mu }^{n}x\right\Vert ^{2}+3\mu
_{t}\left\Vert T_{t}\bar{T}_{\mu }^{n}x-\bar{T}_{\mu }^{n}x\right\Vert ^{2}.
\end{eqnarray*}%
Letting $n$ go to infinity, $\int \Vert T_{t}Rx-Rx\Vert ^{2}d\mu (t)=0.$
Therefore, for each $x\in C$ and $s\in S,$%
\begin{align*}
\Vert T_{s}Rx-Rx\Vert ^{2}& \leq 2\int \Vert T_{s}Rx-T_{t}Rx\Vert ^{2}d\mu
(t)+2\int \Vert T_{t}Rx-Rx\Vert ^{2}d\mu (t) \\
& =2\int \Vert T_{s}Rx-T_{st}Rx\Vert ^{2}d\mu (t)\leq 2k^{2}\int \Vert
Rx-T_{t}Rx\Vert ^{2}d\mu (t)=0.
\end{align*}%
Thus $Rx\in \Fix\mathcal{S}$ for every $x\in C$ and the proof is complete.
\end{proof}

\begin{remark}
Notice that a detailed analysis of the proof of Lemma \ref{holder} gives the
estimation $\alpha =1/(1-\log _{k^{2}-1}k)$ for the exponent and $c=k+8\diam %
C/(2-k^{2})$ for the constant of the H\"{o}lder continuous retraction $R,$ $%
1<k<\sqrt{2}.$
\end{remark}

In particular, the above theorem is applicable to commutative semigroups
since every commutative semigroup is amenable.

\begin{example}
Let $G$ be an unbounded subset of $[0,\infty )$ such that $t+s,t-s\in G$ for
all $t,s\in G$ with $t>s$ (e.g., $G=[0,\infty )$ or $G=\mathbb{N}$). Suppose
that $\mathcal{T}=\{T_{t}:t\in G\}$ is a family of uniformly $k$%
-Lipschitzian mappings from $C$ into $C$ such that $T_{t+s}x=T_{t}\,T_{s}x$
for all $t,s\in G$ and $x\in C$. If $\sup_{t\in G}|T_{t}|<\sqrt{2},$ where $%
|T_{t}|=\inf \{c>0:\Vert T_{t}x-T_{t}y\Vert \leq c\Vert x-y\Vert \ $for\ all$%
\ x,y\in C\}$ denotes the Lipschitz constant of a mapping $T_{t},$ then $\Fix%
\mathcal{T}=\bigcap_{t\in G}\left\{ x\in C:T_{t}x=x\right\} $ is a H\"{o}%
lder continuous retract of $C.$ Notice that we do not assume that the
semigroup $\mathcal{T}$ is continuous in $t$ (i.e., strongly continuous in
the terminology of $C_{0}$-semigroups).
\end{example}

Now consider a more general case.

\begin{example}
Let $S$ be a semitopological semigroup and let $CB(S)$ denote the closed
subalgebra of $\ell ^{\infty }(S)$ consisting of continuous functions. Let $%
LUC(S)$ (resp. $RUC(S)$) be the space of left (resp. right) uniformly
continuous functions on $S,$ i.e., all $f\in CB(S)$ such that the mapping $%
S\ni s\rightarrow l_{s}f$ (resp. $s\rightarrow r_{s}f$) from $S$ to $CB(S)$
is continuous when $CB(S)$ has the sup norm topology. It is known (see \cite%
{Mi}) that $LUC(S)$ and $RUC(S)$ are left and right translation invariant
closed subalgebras of $CB(S)$ containing constants. Notice that if a
uniformly $k$-Lipschitzian semigroup $\mathcal{S}=\left\{ T_{t}:t\in
S\right\} $ is continuous in $t$ on $C$ (i.e., the mapping $S\ni
t\rightarrow T_{t}x$ is continuous for each $x\in C$), then the functions $%
g(t)=\left\langle T_{t}x,y\right\rangle $ and $h(t)=\left\Vert
T_{t}x-y\right\Vert ^{2}$ belong to $RUC(S)$ for every $x\in C$ and $y\in H.$
Indeed (see \cite{La}),%
\begin{align*}
\left\Vert r_{s}g-r_{u}g\right\Vert & =\sup_{t\in S}\left\vert
(r_{s}g)(t)-(r_{u}g)(t)\right\vert =\sup_{t\in S}\left\vert
g(ts)-g(tu)\right\vert \\
& =\sup_{t\in S}\left\vert \left\langle
T_{t}T_{s}x-T_{t}T_{u}x,y\right\rangle \right\vert \leq k\left\Vert
T_{s}x-T_{u}x\right\Vert \left\Vert y\right\Vert
\end{align*}%
and%
\begin{align*}
\left\Vert r_{s}h-r_{u}h\right\Vert & =\sup_{t\in S}\left\vert
h(ts)-h(tu)\right\vert =\sup_{t\in S}\left\vert \left\Vert
T_{ts}x-y\right\Vert ^{2}-\left\Vert T_{tu}x-y\right\Vert ^{2}\right\vert \\
& =\sup_{t\in S}\left\vert (\left\Vert T_{ts}x-y\right\Vert +\left\Vert
T_{tu}x-y\right\Vert )(\left\Vert T_{ts}x-y\right\Vert -\left\Vert
T_{tu}x-y\right\Vert \right\vert \\
& \leq 2\sup_{t\in S}\left\Vert T_{t}x-y\right\Vert \sup_{t\in S}\left\Vert
T_{t}T_{s}x-T_{t}T_{u}x\right\Vert \\
& \leq 2k\sup_{t\in S}\left\Vert T_{t}x-y\right\Vert \left\Vert
T_{s}x-T_{u}x\right\Vert
\end{align*}%
for any $s,u\in S,x\in C$ and $y\in H.$ Thus Theorem \ref{H} is applicable
with $X=RUC(S)$ if we assume that $\mathcal{S}=\left\{ T_{t}:t\in S\right\} $
is (separately) continuous.
\end{example}

Recall that a semitopological semigroup $S$ is said to be left reversible if
any two closed right ideals of $S$ have a non-void intersection. In this
case $(S,\leq )$ is a directed set with the relation $a\leq b$ iff $\left\{
a\right\} \cup \overline{aS}\supset \left\{ b\right\} \cup \overline{bS}.$
In general, the conditions \textquotedblleft $RUC(S)$ has a left invariant
mean\textquotedblright\ and \textquotedblleft $S$ is left
reversible\textquotedblright\ are independent of each other. However, if $%
RUC(S)$ has sufficiently many functions to separate closed sets, then the
former condition implies the latter (see, e.g., \cite[Lemma 3.1]{LaTa}).

Recently, G\'{o}rnicki (cf. \cite[Cor. 14]{Gosem}) proved that if $S$ is
left reversible and $\mathcal{S}=\left\{ T_{t}:t\in S\right\} $ is a
uniformly $k$-Lipschitzian semigroup on $C,$ then the set of fixed points of
$\mathcal{S}$ is a retract of $C.$ It is not clear how to extend Theorem \ref%
{H} to left reversible semigroups. There are many other open problems in
this area. Perhaps the most natural is the following.

\begin{problem}
Under the assumptions of Theorem \ref{GoKi}, is the fixed-point set $\Fix%
\mathcal{S}$ a Lipschitz or H\"{o}lder continuous retract of $C$?
\end{problem}


\begin{thebibliography}{99}
\bibitem{Ar} L. N. Argabright, Invariant means and fixed points: A sequel to
Mitchell's paper, Trans. Amer. Math. Soc. 130 (1968), 127--130.

\bibitem{BeLi} Y. Benyamini, J. Lindenstrauss, Geometric Nonlinear
Functional Analysis, Vol. 1, American Mathematical Society, Providence, RI,
2000.

\bibitem{Br1} R. E. Bruck, Jr., Properties of fixed-point sets of
nonexpansive mappings in Banach spaces, Trans. Amer. Math. Soc. 179 (1973),
251--262.

\bibitem{Br2} R. E. Bruck, Jr., A common fixed point theorem for a commuting
family of nonexpansive mappings, Pacific J. Math. 53 (1974), 59--71.

\bibitem{Da} M. M. Day, Amenable semigroups, Illinois J. Math. 1 (1957),
509--544.

\bibitem{DoRa} D. J. Downing, W. O. Ray, Uniformly Lipschitzian semigroups
in Hilbert space, Canad. Math. Bull. 25 (1982), 210--214.

\bibitem{GoKi} K. Goebel, W. A. Kirk, A fixed point theorem for
transformations whose iterates have uniform Lipschitz constant, Studia Math.
47 (1973), 135--140.

\bibitem{GoJ} J. G\'{o}rnicki, Remarks on the structure of the fixed-point
sets of uniformly Lipschitzian mappings in uniformly convex Banach spaces,
J. Math. Anal. Appl. 355 (2009), 303--310.

\bibitem{Gosem} J. G\'{o}rnicki, The structure of fixed-point sets of
uniformly lipschitzian semigroups, Collect. Math. 63 (2012), 333--344.

\bibitem{IsTa} H. Ishihara, W. Takahashi, Fixed point theorems for uniformly
Lipschitzian semigroups in Hilbert spaces, J. Math. Anal. Appl. 127 (1987),
206--210.

\bibitem{La} A. T.-M. Lau, Semigroup of nonexpansive mappings on a Hilbert
space, J. Math. Anal. Appl. 105 (1985), 514--522.

\bibitem{LaTa} A. T.-M. Lau, W. Takahashi, Fixed point and non-linear
ergodic theorems for semigroups of non-linear mappings, in: Handbook of
Metric Fixed Point Theory, W. A. Kirk, B. Sims (eds.), Kluwer Academic
Publishers, Dordrecht, 2001.

\bibitem{Li} E. A. Lif\v{s}ic, A fixed point theorem for operators in
strongly convex spaces, Vorone\v{z}. Gos. Univ. Trudy Mat. Fak. 16 (1975),
23--28 (in Russian).

\bibitem{Mi} T. Mitchell, Topological semigroups and fixed points, Illinois
J. Math. 14 (1970), 630--641.

\bibitem{MiTa} N. Mizoguchi, W. Takahashi, On the existence of fixed points
and ergodic retractions for Lipschitzian semigroups in Hilbert spaces,
Nonlinear Anal. 14 (1990), 69--80.

\bibitem{PeFe} V. P\'{e}rez Garc\'{\i}a, H. Fetter Nathansky, Fixed points
of periodic mappings in Hilbert spaces, Ann. Univ. Mariae Curie-Sk\l odowska
Sect. A 64 (2010), 37--48.

\bibitem{SeWi} E. S\c{e}d{\l }ak, A. Wi\'{s}nicki, On the structure of
fixed-point sets of uniformly Lipschitzian mappings, Topol. Methods
Nonlinear Anal. 30 (2007), 345--350.

\bibitem{Ta} W. Takahashi, A nonlinear ergodic theorem for an amenable
semigroup of nonexpansive mappings in a Hilbert space, Proc. Amer. Math.
Soc. 81 (1981), 253--256.
\end{thebibliography}
\end{document}